\documentclass[12pt]{amsart} 

\usepackage{amsmath} \usepackage{amssymb} \usepackage{amsthm}
\usepackage{amscd} 
\usepackage[all]{xy} 
\usepackage{enumerate}

\def\a{{\mathfrak{a}}} 

\def\F{{\mathbb{F}}} 
\def\I{{\mathcal{I}}} 
\def\J{{\mathcal{J}}} 
\def\m{{\mathfrak{m}}}
\def\n{{\mathfrak{n}}} 
\def\Z{{\mathbb{Z}}}

\def\sO{{\mathcal{O}}}

\def\Q{{\mathbb{Q}}} 
\def\R{{\mathbb{R}}} 
\def\C{{\mathbb{C}}}
 
\def\Ker{{\mathrm{Ker\;}}}

\def\Ann{{\mathrm{Ann}}}
\def\Spec{{\mathrm{Spec\; }}}

\def\adj{{\mathrm{adj}}}
\def\Div{{\mathrm{div}}}
\def\ord{{\mathrm{ord}}} 

\theoremstyle{plain}
\newtheorem{thm}{Theorem}[section] 
\newtheorem{cor}[thm]{Corollary}
\newtheorem{prop}[thm]{Proposition}
\newtheorem{conj}[thm]{Conjecture}
\newtheorem{mainthm}{Theorem}
 
\newtheorem{deflem}[thm]{Definition-Lemma}
\newtheorem{lem}[thm]{Lemma}
\theoremstyle{definition} 
\newtheorem{defn}[thm]{Definition}
\newtheorem{eg}[thm]{Example} 
\theoremstyle{remark}
\newtheorem{rem}[thm]{Remark}

\newtheorem*{cl}{Claim}
\newtheorem{cln}{Claim}
\newtheorem*{acknowledgement}{Acknowledgments}

\title{Adjoint ideals along closed subvarieties\\ of higher codimension}
\author{Shunsuke Takagi}
\address{Department of Mathematics, Kyushu University, 
6-10-1 Hakozaki, Higashi-ku, Fukuoka, 812-8581 Japan}
\email{stakagi@math.kyushu-u.ac.jp}

\begin{document}
\tolerance = 9999

\begin{abstract}
In this paper, we introduce a notion of adjoint ideal sheaves along closed subvarieties of higher codimension and study its local properties using characteristic $p$ methods.  
When $X$ is a normal Gorenstein closed subvariety of a smooth complex variety $A$, we formulate a restriction property of the adjoint ideal sheaf $\adj_X(A)$ of $A$ along $X$ involving the l.c.i.~ideal sheaf  $\mathcal{D}_X$ of $X$. 
The proof relies on a modification of generalized test ideals of Hara and Yoshida \cite{HY}. 
\end{abstract}

\maketitle
\markboth{Shunsuke Takagi}{Adjoint ideals along closed subvarieties of higher codimension}
\section*{Introduction} 
The adjoint ideal sheaf along a divisor $D$ on a complex variety $V$ is a modification of the multiplier ideal sheaf associated to $D$, and it encodes much information on the singularities of $D$. 
It recently turned out that it is a powerful tool in birational geometry and has several applications, such as the study of singularities of ample divisors of low degree on abelian varieties by Ein-Lazarsfeld \cite{EL} and Debarre-Hacon \cite{DH}, inversion of adjunction on log canonicity proved by Kawakita \cite {Ka2}, and the boundedness of pluricanonical maps of varieties of general type proved by Hacon-M$\mathrm{{}^c}$Kernan \cite{HM} and Takayama \cite{Takayama}. 
In this paper, we introduce a notion of adjoint ideal sheaves along closed subvarieties of higher codimension and study its local properties using characteristic $p$ methods. We hope that our adjoint ideal sheaves lead to further applications. 

Let $A$ be a smooth complex variety and $Y=\sum_{i=1}^m t_i Y_i$ be a formal combination, where the $t_i$ are positive real numbers and the $Y_i$ are proper closed subschemes of $A$. 
Let $X$ be a reduced closed subscheme of pure codimension $c$ of $A$ such that no components of $X$ are contained in the support of any $Y_i$. 
Suppose that $\pi: \widetilde{A} \to A$ is a log resolution of $(A, X+Y)$ and $E:=\sum_{j=1}^sE_j$ is smooth, where $E_1, \dots, E_s$ are all the irreducible divisors on $\widetilde{A}$ ``dominating" a component of $X$.  
If $K_{\widetilde{A}/A}$ is the relative canonical divisor of $\pi$, then we define the \textit{adjoint ideal sheaf} $\adj_X(A,Y)$ associated to the pair $(A,Y)$ along $X$ by
$$\adj_X(A,Y):=\pi_*\sO_{\widetilde{A}}(K_{\widetilde{A}/A}-c \ \pi^{-1}(X)-\lfloor \pi^{-1}(Y) \rfloor +E),$$
where $\pi^{-1}(X)$ and $\pi^{-1}(Y):=\sum_{i=1}^m t_i \pi^{-1}(Y_i)$ are the scheme theoretic inverse images of $X$ and $Y$, respectively (see Definition \ref{adjoint def} for the precise definition of the adjoint ideal sheaf $\adj_X(A,Y)$). 
We say that $(A,Y)$ is \textit{purely log terminal} (plt, for short) along $X$ if $\adj_X(A,Y)=\sO_A$. 
When $X$ is a divisor, our definitions coincide with the definitions of usual plt pairs and adjoint ideal sheaves. 
In order to study local properties of our adjoint ideal sheaves using characteristic $p$ methods, we consider a modification of generalized test ideals of Hara and Yoshida. 

Let $(R,\m)$ be a Noetherian local domain of characteristic $p>0$ and $\underline{\a}^{\underline{t}}=\prod_{i=1}^m \a_i^{t_i}$ be a formal combination, where the $\a_i \subseteq R$ are nonzero ideals and the $t_i$ are positive real numbers. 
Hara-Yoshida \cite{HY} introduced notions of tight closure for the pair $(R,\underline{\a}^{\underline{t}})$,  called $\underline{\a}^{\underline{t}}$-tight closure, and the corresponding test ideal $\widetilde{\tau}(\underline{\a}^{\underline{t}})$. 
They then proved that the multiplier ideal sheaf coincides, after reduction to characteristic $p \gg 0$, with their generalized test ideal. 
In this paper, we define a notion of tight closure for a triple $(R,I, \underline{\a}^{\underline{t}})$, called $(I, \underline{\a}^{\underline{t}})$-tight closure, where $I \subseteq R$ is an unmixed ideal of height $c$ such that the $\a_i$ are not contained in any minimal prime ideal of $I$: the $(I,\underline{\a}^{\underline{t}})$-tight closure $J^{*(I,\underline{\a}^{\underline{t}})}$ of an ideal $J \subseteq R$ is the ideal consisting of all elements $x \in R$ for which there exists $\gamma \in R$ not in any minimal prime of $I$ such that
$$\gamma I^{c(q-1)}\a_1^{\lceil t_1q \rceil} \cdots \a_m^{\lceil t_mq \rceil} x^q \subseteq J^{[q]}$$
for all large $q=p^e$, where $J^{[q]}$ is the ideal generated by the $q^{\rm th}$ powers of all elements of $J$. 
If $N \subseteq M$ are $R$-modules, then the $(I,\underline{\a}^{\underline{t}})$-tight closure $N_M^{*(I,\underline{\a}^{\underline{t}})}$ of $N$ in $M$ is defined similarly (see Definition \ref{test ideal} for the detail). 
We then define the generalized test ideal $\widetilde{\tau}_I(\underline{\a}^{\underline{t}})$ along $I$
to be the annihilator ideal of the $(I,\underline{\a}^{\underline{t}})$-tight closure $0^{*(I,\underline{\a}^{\underline{t}})}_{E_R(R/\m)}$ of the zero submodule in the injective hull $E_R(R/\m)$ of the residue field of $R$. 
When $I=R$,  $(I, \underline{\a}^{\underline{t}})$-tight closure coincides with  $\underline{\a}^{\underline{t}}$-tight closure and the generalized test ideal $\widetilde{\tau}_I(\underline{\a}^{\underline{t}})$ along $I$ is nothing but $\widetilde{\tau}(\underline{\a}^{\underline{t}})$. 
We conjecture that the ideal $\widetilde{\tau}_I(\underline{\a}^{\underline{t}})$ corresponds to the adjoint ideal sheaf $\adj_X(A,Y)$, and we obtain some partial results (Theorems \ref{I < adj} and \ref{correspond}). 
We use them to prove a restriction formula of our adjoint ideal sheaves. 

Kawakita \cite{Ka} and Ein-Musta\c{t}\v{a} \cite{EM} introduced an ideal sheaf, called the l.c.i.~defect ideal sheaf, which measures how far a variety is from being locally a complete intersection. 
They then proved a comparison of minimal log discrepancies of a variety $X$ and its ambient space $A$ with a boundary corresponding to the l.c.i.~defect ideal sheaf $\mathcal{D}_X$ of $X$.  
Their result inspires us to formulate a restriction property of the adjoint ideal sheaf $\adj_X(A,Y)$ involving the l.c.i.~defect ideal sheaf $\mathcal{D}_X$ of $X$. 

\renewcommand{\themainthm}{3.1}
\begin{mainthm}
Let $A$ be a smooth complex variety and $Y=\sum_{i=1}^m t_i Y_i$ be a formal combination, where the $t_i$ are positive real numbers and the $Y_i$ are proper closed subschemes of $A$.
If $X$ is a normal Gorenstein closed subvariety of codimension $c$ of $A$ which is not contained in the support of any $Y_i$, then 
$$\J(X,V(\mathcal{D}_X)+Y|_X)=\adj_X(A,Y) \sO_X, $$
where 
$\J(X,V(\mathcal{D}_X)+Y|_X)$ is the multiplier ideal sheaf associated to the pair $(X,V(\mathcal{D}_X)+Y|_X)$ $($see Definition \ref{codimension one} for the definition of multiplier ideal sheaves$)$.
\end{mainthm}

By making use of the partial correspondence between the adjoint ideal sheaf $\adj_X(A,Y)$ and the generalized test ideal $\widetilde{\tau}_I(\underline{\a}^{\underline{t}})$, Theorem \ref{restriction} can be reduced to a purely algebraic problem on some ideals of a ring of characteristic $p>0$. 
We then solve the problem using the linkage theory of Peskine and Szpiro \cite{PS}.  

As a corollary of Theorem \ref{restriction}, we obtain a characterization of being plt along a Gorenstein closed subvariety in terms of Frobenius splitting (Corollary \ref{purely F-regular vs plt}). 

\section{Multiplier ideals and Adjoint ideals}
In this section, we first recall the definitions of multiplier ideal sheaves and adjoint ideal sheaves along divisors (our main references are \cite{KM} and \cite{La}), and then we introduce a notion of adjoint ideal sheaves along closed subvarieties of higher codimension. 

Let $X$ be a $d$-dimensional $\Q$-Gorenstein normal variety over a field $k$ of characteristic zero and $Y=\sum_{i=1}^m t_iY_i$ be a formal combination, where the $t_i$ are real numbers and the $Y_i$ are proper closed subschemes of $X$. 
Since $X$ is normal, we have a Weil divisor $K_X$ on $X$, uniquely determined up to linear equivalence, such that $\sO_X(K_X) \cong i_*\Omega^d_{X_{\rm reg}}$ where $i:X_{\rm reg} \hookrightarrow X$ is the inclusion of the nonsingular locus. 
Moreover, since $X$ is $\Q$-Gorenstein, there exists a positive integer $r$ such that $r K_X$ is a Cartier divisor. 

Let $E$ be a \textit{divisor over} $X$, that is, $E$ is an irreducible divisor on some normal variety $X'$ with a birational morphism $f:X' \to X$.
We identify two divisors over $X$ if they correspond to the same valuation of the function field $k(X)$. 
The \textit{center} of $E$ is the closure of $f(E)$ in $X$, denoted by $c_X(E)$. 
If $Z$ is a closed subscheme of $X$, then we define $\ord_E(Z)$ as follows: we may assume that the scheme theoretic inverse image $f^{-1}(Z)$ is a divisor. 
Then $\ord_E(Z)$ is the coefficient of $E$ in $f^{-1}(Z)$. 
We put $\ord_E(Y):=\sum_{i=1}^m t_i \ord_E(Y_i)$ and define $\ord_E(K_{-/X})$ as the coefficient of $E$ in the relative canonical divisor $K_{X'/X}$ of $f$. 
Recall that $K_{X'/X}$ is the unique $\Q$-divisor supported on the exceptional locus of $f$ such that $r K_{X'/X}$ is linearly equivalent to $r K_{X'}-f^*(rK_X)$. 
Then the \textit{log discrepancy} $a(E;X,Y)$ of $(X,Y)$ with respect to $E$ is 
$$a(E;X,Y):=\mathrm{ord}_E(K_{-/X})-\ord_E(Y)+1.$$
If $W$ is a closed subset of $X$, then the \textit{minimal log discrepancy} $\mathrm{mld}(W;X,Y)$ of $(X,Y)$ along $W$ is defined by
$$\mathrm{mld}(W;X,Y):=\inf\{a(E;X,Y) \mid \textup{$E$ is a divisor over $X$, } c_X(E) \subseteq W\}.$$

\begin{defn}\label{sing pairs}
Let the notation be the same as above. 
\renewcommand{\labelenumi}{(\roman{enumi})}
\begin{enumerate}
\item We say that the pair $(X,Y)$ is \textit{Kawamata log terminal} (klt, for short) if $\mathrm{mld}(X;X,Y)>0$. 
Since a resolution of singularities is obtained by blowing up subvarieties in the singular locus, this condition is equivalent to saying that $\mathrm{mld}(X_{\rm sing} \cup \bigcup_{i=1}^m Y_i;X,Y)>0$, where $X_{\rm sing}$ is the singular locus of $X$. 
\item Let $D$ be a reduced Cartier divisor on $X$ such that no components of $D$ are contained in the support of any $Y_i$. Then we say that $(X,Y)$ is \textit{purely log terminal} (plt, for short) along $D$ if 
$a(E;X,D+Y)>0$ for all divisors $E$ over $X$ dominating no components of $D$.  
\end{enumerate}
\end{defn}

Suppose that $(X,Y)$ is a pair as above.
A \textit{log resolution} of the pair $(X,Y)$ is a proper birational morphism $\pi: \widetilde{X} \to X$ with $\widetilde{X}$ nonsingular such that all the scheme theoretic inverse images $\pi^{-1}(Y_i)$ are divisors and in addition $\bigcup_{i=1}^m \mathrm{Supp}\; \pi^{-1}(Y_i) \cup \mathrm{Exc}(\pi)$ is a simple normal crossing divisor. 
The existence of log resolutions is guaranteed by Hironaka's desingularization theorem \cite{Hi}. 

\begin{defn}[\textup{\cite[Definition 9.3.60]{La}}]\label{codimension one}
Let the notation be the same as above. 
\renewcommand{\labelenumi}{(\roman{enumi})}
\begin{enumerate}
\item 
Fix a log resolution $\pi: \widetilde{X} \to X$ of $(X,Y)$. 
The \textit{multiplier ideal sheaf} $\J(X,Y)$ associated to the pair $(X,Y)$ is
$$\J(X,Y)=\pi_*\sO_{\widetilde{X}}(\lceil K_{\widetilde{X}/X}-\sum_{i=1}^m t_i \pi^{-1}(Y_i) \rceil) \subseteq \sO_X.$$
\item 
Let $D$ be a reduced Cartier divisor on $X$ such that no components of $D$ are contained in the support of any $Y_i$. Fix a log resolution $\pi: \widetilde{X} \to X$ of $(X,D+Y)$ so that the strict transform $\pi^{-1}_*D$ of $D$ is nonsingular (but possibly disconnected). 
Then the \textit{adjoint ideal sheaf} $\adj_{D}(X,Y)$ associated to the pair $(X,Y)$ along $D$ is $$\adj_{D}(X,Y)=\pi_*\sO_{\widetilde{X}}(\lceil K_{\widetilde{X}/X}-\sum_{i=1}^m t_i\pi^{-1}(Y_i)-\pi^*D+\pi^{-1}_*D \rceil) \subseteq \sO_X.$$
We denote this ideal sheaf simply by $\adj_{D}(X)$ when $Y=0$. 
\end{enumerate}
\end{defn}

\begin{rem}\label{adjoint rem}
\begin{enumerate}
\item (cf. \cite[Theorem 9.2.18]{La}) 
$\J(X,Y)$ and $\adj(X,Y)$ are independent of the choice of the log resolution $\pi$ used to define them (see also Lemma \ref{adjoint basic} (1)). 
\item 
The pair $(X,Y)$ is klt (resp. plt along $D$) if and only if $\J(X,Y)=\sO_X$ (resp. $\adj_D(X,Y)=\sO_X$). 
\item 
(\cite[Example 9.3.49]{La}) Suppose that $X$ is an affine variety and $I$ is a nonzero ideal of $\sO_X$. 
Choose a general element $f$ in $I$ so that $\Div_X(f)$ is reduced and no components of $\Div_X(f)$ are contained in the support of any $Y_i$. 
Then 
$$\J(X,V(I)+Y)=\adj_{\Div_X(f)}(X,Y)$$
(see also Claim 1 in the proof of Theorem \ref{restriction}). 
\end{enumerate}
\end{rem}

An analogue of local vanishing theorem \cite[Theorem 9.4.1]{La} holds for the adjoint ideal sheaf $\adj_D(X,Y)$ along a divisor $D$. 
\begin{prop}
Let the notation be the same as in Definition \ref{codimension one} $\textup{(ii)}$. 
Then for all $i>0$,
$$R^i\pi_*\sO_{\widetilde{X}}(\lceil K_{\widetilde{X}/X}-\pi^{-1}(Y)-\pi^*D+\pi^{-1}_*D \rceil)=0.$$ 
\end{prop}
\begin{proof}
Set $B:=\lceil K_{\widetilde{X}/X}-\pi^{-1}(Y)-\pi^*D \rceil$ and $\widetilde{D}:=\pi^{-1}_*D$. 
Let $\nu:D^{\nu} \to D$ be the normalization of $D$, $\mu: \widetilde{D} \to D^{\nu}$ be the induced morphism and $\pi_D: \widetilde{D} \to D$ be the composite morphism. 
Then there exists an effective $\Q$-divisor ${\rm Diff}_{D^{\nu}}(0)$ on $D^{\nu}$, called the {\it different} of the zero divisor on $D^{\nu}$ (see \cite[\S 3]{Sh} for details), such that $K_{D^{\nu}}+{\rm Diff}_{D^{\nu}}(0)$ is $\Q$-Cartier and $K_{D^{\nu}}+{\rm Diff}_{D^{\nu}}(0)= \nu^*((K_X +D)|_D)$.
Now we have the following exact sequence
$$0 \to \sO_{\widetilde{X}}(B) \to \sO_{\widetilde{X}}(B+\widetilde{D}) \to \sO_{\widetilde{D}}(\lceil K_{\widetilde{D}}-\mu^*(K_{D^{\nu}}+{\rm Diff}_{D^{\nu}}(0))-\pi_D^{-1}(Y|_D)\rceil) \to 0.$$
It follows from Kawamata-Viehweg vanishing theorem that 
$$R^i \pi_*\sO_{\widetilde{X}}(B)=R^i {\pi_D}_*\sO_{\widetilde{D}}(\lceil K_{\widetilde{D}}-\mu^*(K_{D^{\nu}}+{\rm Diff}_{D^{\nu}}(0))-\pi_D^{-1}(Y|_D)\rceil)=0$$ for all $i>0$. 
Thus, we have $R^i \pi_*\sO_{\widetilde{X}}(B+\widetilde{D})=0$ for all $i>0$. 
\end{proof}

\begin{eg}
Let $X=\C^2=\Spec \C[x,y]$ be the two-dimensional affine space and let $D=(x^3+y^5=0) \subseteq X$. 
Then $\adj_D(X)=(x^2, xy, y^3)$, whereas $\J(X,D)=(x^3+y^5)$.  
\end{eg}

When the ambient variety is smooth, we can generalize the notion of adjoint ideal sheaves to the higher codimension case. 

Let $A$ be a nonsingular variety over a field $k$ of characteristic zero and $Y=\sum_{i=1}^m t_i Y_i$ be a formal combination, where the $t_i$ are positive real numbers and the $Y_i$ are proper closed subschemes of $A$. 
Let $X$ be a reduced closed subscheme of pure codimension $c$ of $A$ such that  no components of $X$ are contained in the support of any $Y_i$. 
Let $f:A':=\mathrm{Bl}_X A \to A$ be the blowing-up of $A$ along $X$ and $E_1, \dots, E_s$ be all the components of the exceptional divisor of $f$ dominating an irreducible component of $X$. 
Fix a log resolution $g:\widetilde{A} \to A'$ of $(A', f^{-1}(X)+ f^{-1}(Y))$ such that  $\sum_{j=1}^s g^{-1}_*E_j$ is nonsingular (but possibly disconnected), and put $\pi:=f \circ g:\widetilde{A} \to A$.  
\begin{defn}\label{adjoint def}
In the above situation, the \textit{adjoint ideal sheaf} $\adj_X(A,Y)$ associated to the pair $(A,Y)$ along $X$ is 
$$\adj_{X}(A,Y)=\pi_*\sO_{\widetilde{A}}(K_{\widetilde{A}/A}-\sum_{i=1}^m \lfloor t_i \pi^{-1}(Y_i) \rfloor -c \ \pi^{-1}(X)+\sum_{j=1}^s g^{-1}_*E_j) \subseteq \sO_A.$$
We denote this ideal sheaf simply by $\adj_{X}(A)$ when $Y=0$. 
We say that $(A,Y)$ (resp. $A$) is \textit{purely log terminal} (plt, for short) along $X$ if $\adj_X(A,Y)=\sO_A$ (resp. $\adj_X(A)=\sO_A$). 
When $X$ is a divisor, these definitions coincide with those given in Definition \ref{sing pairs} (ii) and Definition \ref{codimension one} (ii). 
\end{defn}

\begin{lem}\label{adjoint basic}
Let the notation be as in Definition \ref{adjoint def}. 
\renewcommand{\labelenumi}{$(\arabic{enumi})$}
\begin{enumerate}
\item
The adjoint ideal sheaf $\adj_X(A,Y)$ is independent of the choice of the log resolution used to define it. 
\item 
$(A,Y)$ is plt along $X$ if and only if 
$$\mathrm{mld}(X_{\mathrm{sing}} \cup \bigcup_{i=1}^m Y_i;A,cX+Y)>0,$$
where $X_{\mathrm{sing}}$ is the singular locus of $X$. 
More generally, the adjoint ideal sheaf $\adj_X(A,Y)$ is an ideal sheaf of $X$ whose sections over an open subset $U$ are those $\varphi \in \sO_X(U)$ such that for every divisor $E$ over $X$ whose center intersects $U$ and is contained in $X_{\mathrm{sing}} \cup \bigcup_{i=1}^m Y_i$, 
$$\mathrm{ord}_E(\varphi)+a(E;A, cX+Y) >0.$$
\end{enumerate}
\end{lem}
\begin{proof}
Let  $f: A' \to A$ be the blowing-up of $A$ along $X$ and $E_1, \dots, E_s$ be all the components of the exceptional divisor of $f$ dominating an irreducible component of $X$. 
Put $E=E_1+\cdots+E_s$. 

(1) The proof is essentially the same as that of \cite[Theorem 9.2.18]{La}. 
We consider a sequence of morphisms $V \xrightarrow{\nu} \widetilde{A} \xrightarrow{\pi} A$, where $\pi$ is a log resolution of $(A,X+Y)$ such that the strict transform $\widetilde{E}$ of $E$ is nonsingular and $\nu$ is a log resolution of $(\widetilde{A}, \pi^{-1}(X)+\pi^{-1}(Y))$.  

\begin{cl}
Let $D$ be a reduced disconnected divisor on $\widetilde{A}$ with simple normal crossing support and $B$ be an $\R$-divisor on $\widetilde{A}$ with simple normal crossing support which has no common components with $D$. 
Suppose that $\mu:W \to \widetilde{A}$ is a log resolution of $D+B$. Then
$$\mu_*\sO_{W}(K_{W/\widetilde{A}}-\lfloor \mu^*(D+B) \rfloor +\mu^{-1}_*D)=\sO_{\widetilde{A}}(-\lfloor B \rfloor).$$ 
\end{cl}
\begin{proof}[Proof of Claim]
It follows from the projection formula that if the assertion holds for a given $\R$-divisor $B$, then it holds also for $B+B'$ whenever $B'$ is an integral divisor on $\widetilde{A}$. 
Therefore, we may assume that $\lfloor B \rfloor=0$. 
Setting $\Delta=D+B$, we have 
$$K_{W}+\mu^{-1}_*\Delta=\mu^*(K_{\widetilde{A}}+\Delta)+\sum_F (a(F;\widetilde{A},\Delta)-1)F,$$
where $F$ runs through all $\mu$-exceptional prime divisors on $W$. Then
$$K_{W/\widetilde{A}}- \lfloor \mu^*\Delta \rfloor +\mu^{-1}_*D=\lceil K_{W/\widetilde{A}}-\mu^*\Delta+\mu^{-1}_*\Delta \rceil=\sum_F \lceil a(F;\widetilde{A},\Delta)-1 \rceil F,$$
because $\lfloor B \rfloor=0$. 
Since $D$ is disconnected, by \cite[Corollary 2.31 (3)]{KM}, one has $a(F;\widetilde{A},\Delta)>0$ for all $\mu$-exceptional prime divisors $F$. This completes the proof.  
\end{proof}

By the above claim, we know that 
$${\nu}_*\sO_{V}(K_{V/\widetilde{A}}-\lfloor {\nu}^{-1}\pi^{-1}(Y) \rfloor -{\nu}^*\widetilde{E}+{\nu}^{-1}_*\widetilde{E})=\sO_{\widetilde{A}}(-\lfloor \pi^{-1}(Y) \rfloor). $$
Then, setting $h:=\pi \circ \nu$, one finds using the projection formula:
\begin{align*}
&h_*\sO_{V}(K_{V/\widetilde{A}}-\lfloor h^{-1}(Y) \rfloor -c \ h^{-1}(X)+\nu^{-1}_*\widetilde{E})\\
=&\pi_*\nu_*\left( \nu^*\sO_{\widetilde{A}}(K_{\widetilde{A}/A}-c \ \pi^{-1}(X)+\widetilde{E}) \otimes \sO_V(K_{V/\widetilde{A}}-\lfloor \nu^{-1}\pi^{-1}(Y) \rfloor -\nu^*\widetilde{E}+\nu^{-1}_*\widetilde{E})\right)\\
=&\pi_* \left(\sO_{\widetilde{A}}(K_{\widetilde{A}/A}-c \ \pi^{-1}(X)+\widetilde{E}) \otimes \nu_*\sO_{V}(K_{V/\widetilde{A}}-\lfloor \nu^{-1}\pi^{-1}(Y) \rfloor -\nu^*\widetilde{E}+\nu^{-1}_*\widetilde{E}) \right)\\
=&\pi_*\sO_{\widetilde{A}}(K_{\widetilde{A}/A}-\lfloor \pi^{-1}(Y) \rfloor-c \ \pi^{-1}(X)+\widetilde{E}).
\end{align*}
In other words, we obtain the same adjoint ideal sheaf working from $h$ as working from $\pi$. 
Since any two resolutions can be dominated by a third, the assertion follows. 

(2) Since no components of $X$ are contained in the support of any $Y_i$, $Y$ does not contribute to $a(E_i;A,cX+Y)$. By \cite[Lemma 2.29]{KM}, one has $a(E_i;A, cX+Y)=a(E_i;A,cX)=0$ for all $i=1, \dots, s$. 
We have already seen in (1) that the adjoint ideal sheaf is independent of the choice of the log resolution used to define it.  
Since 
\begin{align*}
K_{\widetilde{A}/A}-\lfloor \pi^{-1}(Y) \rfloor -c \ \pi^{-1}(X)+g^{-1}_*E =& \hspace{-0.5em} \sum_{F:\textup{divisor on $\widetilde{A}$}} \hspace{-0.5em} \lceil a(F;A,cX+Y)-1 \rceil F+g^{-1}_*E\\
=& \sum_{F \ne g^{-1}_*E_i}  \lceil a(F;A,cX+Y)-1 \rceil  F
\end{align*}
for every log resolution $g:\widetilde{A} \to A'$ of $(A',f^{-1}(X)+f^{-1}(Y))$ where $\pi=f \circ g: \widetilde{A} \to A$, the pair $(A,Y)$ is plt  along $X$ if and only if $a(F;A,c X+Y)>0$ for every divisor $F$ over $A$ dominating no components of $E$.   
This implies that if $(A,Y)$ is plt  along $X$, then $\mathrm{mld}(X_{\mathrm{sing}} \cup \bigcup_{i=1}^m Y_i;A,cX+Y)>0$. 
The converse follows from the fact that a log resolution of $(A', f^{-1}(X)+f^{-1}(Y))$ is obtained by blowing up subvarieties in $f^{-1}(X_{\mathrm{sing}} \cup \bigcup_{i=1}^n Y_i)$. 
We can prove the general case similarly.  
\end{proof}

\begin{eg}\label{adjoint example}
(1) Suppose that $X$ is locally a complete intersection variety, and consider the blowing-up $f:A'=\mathrm{Bl}_X A \to A$ of $A$ along $X$. 
Then the exceptional divisor $E:=f^{-1}(X)$ is a projective bundle over $X$. 
In particular, $E$ is normal and locally a complete intersection, and therefore, so is $A'$. 
By Hironaka's embedded resolution of singularities \cite{Hi}, there exists a log resolution $g:\widetilde{A} \to A'$ of $(A', E)$ which is an isomorphism over the complement of a proper closed subset of $E$. 
Let $\widetilde{E}$ be the strict transform of $E$ on $\widetilde{A}$, and put $\pi:=f \circ g: \widetilde{A} \to A$. 
Since 
$$K_{\widetilde{A}/A}-c \ \pi^{-1}(X)+ \widetilde{E}=K_{\widetilde{A}/A'}-g^{*}E+ \widetilde{E},$$
$A$ is plt along $X$ if and only if $A'$ is plt along $E$. 
Since $A' \setminus E$ is nonsingular, a result of Koll\'ar \cite[Theorem 5.50]{KM} says that  $A'$ is plt along $E$ if and only if $E$ is klt. 
This means that $X$ is klt, because $E$ is locally a product of $X$ and an affine space.
It therefore follows that $A$ is plt along $X$ if and only if $X$ is klt. 
That is, $\adj_X(A)$ defines the non-klt locus of $X$. 

(2) Let $X=\frac{1}{3}(1,1,1)$ be the quotient of $\C^3=\Spec \C[x_1, x_2, x_3]$ by the action of $\Z/ 3\Z$ given by $x_i \mapsto \xi x_i$, where $\xi$ is a primitive cubic root of unity. 
$X$ can be embedded into $A:=\C^{10}$, and we will compute the ideal sheaf $\adj_X(A)$. 
Let $\pi_1: A_1 \to A$ be the blowing-up of $A$ at the origin with exceptional divisor $E_1$ (we use the same letter for its strict transform). 
Then the weak transform $X_1$ of $X$ is nonsingular. Next, let $\pi_2:A_2 \to A_1$ be the blowing-up of $A_1$ along $X_1$ with exceptional divisor $E_2$.  
Setting $\pi:=\pi_1 \circ \pi_2:A_2 \to A$, we have $K_{A_2/A}=K_{A_2/A_1}+\pi_2^*K_{A_1/A}=9E_1+6E_2$ and $\pi^{-1}(X)=2E_1+E_2$. Thus,
$$\adj_X(A)=\pi_*\sO_{A_2}(K_{A_2/A}-7 \pi^{-1}(X)+E_2)=\pi_*\sO_{A_2}(-5E_1)=\m_{A, 0}^5,$$
where $\m_{A, 0} \subseteq \sO_{A}$ is the maximal ideal sheaf of the origin. 

(3) Let $X$ be the quotient of $(x^2+y^3+z^6=0) \subset \C^3$ by the action of $\Z/ 5\Z$ given by $x \mapsto \xi^3  x$, $y \mapsto \xi^2  y$ and $z \mapsto \xi  z$, where $\xi$ is a primitive quintic root of unity. Then $X$ can be embedded into $A:=\C^5$, and by an argument similar to that of (2), we have $\adj_X(A)=\m_{A,0}^2$, where $\m_{A, 0} \subseteq \sO_{A}$ is the maximal ideal sheaf of the origin. 
\end{eg}

\section{A modification of generalized test ideals}
In this section, we consider a modification of generalized test ideals of Hara and Yoshida \cite{HY},  which conjecturally corresponds to our adjoint ideal sheaf. 

Throughout this paper, all rings are Noetherian commutative rings with identity. 
For an integral domain $R$ and an unmixed ideal $I$ of $R$, we denote by $R^{\circ, I}$ the set of elements of $R$ that are not in any minimal prime ideal of $I$. 

 Let $R$ be an integral domain of characteristic $p>0$. 
For an ideal $J$ of $R$ and a power $q$ of $p$, we denote by $J^{[q]}$ the ideal of $R$ generated by the $q^{\rm th}$ powers of all elements of $J$. 
Let $F:R \to R$ be the Frobenius map, that is, the ring homomorphism sending $x$ to $x^p$. 
The ring $R$ viewed as an $R$-module via the $e$-times iterated Frobenius map $F^e \colon R \to R$ is denoted by ${}^e\! R$.
Since $R$ is reduced, $F^e:R \to {}^e\! R$ is identified with the natural inclusion map $R \hookrightarrow R^{1/p^e}$.
We say that $R$ is {\it F-finite} if ${}^1\! R$ (or $R^{1/p}$) is a finitely generated $R$-module. 
For example, any algebra essentially of finite type over a perfect field is F-finite. 

\begin{defn}[\textup{[Ta1, Definition 3.1]}]\label{F-pair}
Let $R$ be an F-finite domain of characteristic $p>0$ and $\underline{\a}^{\underline{t}}=\prod_{i=1}^m \a_i^{t_i}$ be a formal combination, where the $\a_i$ are nonzero ideals of $R$ and the $t_i$ are positive real numbers.  
\renewcommand{\labelenumi}{(\roman{enumi})}
\begin{enumerate}
\item 
The pair $(R, \underline{\a}^{\underline{t}})$ is said to be {\it strongly F-regular} if for every $\gamma \in R^{\circ}$, there exist $q=p^e$ and $\delta \in \a_1^{\lceil t_1q \rceil} \cdots \a_m^{\lceil t_mq \rceil}$ such that $(\gamma \delta)^{1/q}R \hookrightarrow R^{1/q}$ splits as an $R$-module homomorphism. 
\item 
Let $I \subseteq R$ be an unmixed ideal of height $c$ and suppose that $\a_i \cap R^{\circ, I} \ne \emptyset$ for all $i=1, \dots, m$. 
Then the pair $(R, \underline{\a}^{\underline{t}})$ is said to be {\it purely F-regular} along $I$ if for every $\gamma \in R^{\circ,I}$  there exist $q=p^e$ and $\delta \in I^{c(q-1)}\a_1^{\lceil t_1q \rceil} \cdots \a_m^{\lceil t_mq \rceil}$ such that $(\gamma \delta)^{1/q}R \hookrightarrow R^{1/q}$ splits as an $R$-module homomorphism. 
We say that $R$ is purely F-regular along $I$ if so is the pair $(R,R^1)$. 
\end{enumerate}
\end{defn}

Let $R$ be an integral domain of characteristic $p > 0$ and $M$ be an $R$-module. 
For each $q=p^e$, we denote $\F^{e}(M)= \F_{R}^{e}(M) :=  {}^e\! R \otimes_R M$ and regard it as an $R$-module by the action of $R={}^e R$ from the left. 
Then we have the $e$-times iterated Frobenius map $F_M^{e} \colon M \to 
\F^{e}(M)$ induced on $M$. The image of an element $z \in M$ via this map is denoted by $z^q:= F_M^{e}(z) \in \F^{e}(M)$. 
For an $R$-submodule $N$ of $M$, we denote by $N^{[q]}_{M}$ the 
image of the induced map $\F^{e}(N) \to \F^{e}(M)$. 
If $I $ is an ideal of $R$, then $I^{[q]}_{R}=I^{[q]}$. 

Now we introduce a new generalization of tight closure and the corresponding test ideal. 
\begin{defn}\label{test ideal}
Let $R$ be an excellent domain of characteristic $p>0$ and $I \subseteq R$ be an unmixed ideal of height $c$.
Let $\underline{\a}^{\underline{t}}=\prod_{i=1}^m \a_i^{t_i}$ be a formal combination, where the $\a_i$ are ideals of $R$ such that $\a_i \cap R^{\circ, I} \ne \emptyset$ and the $t_i$ are positive real numbers. 
\renewcommand{\labelenumi}{(\roman{enumi})}
\begin{enumerate}
\item If $N \subseteq M$ are $R$-modules, then the \textit{$(I,\underline{\a}^{\underline{t}})$-tight closure} $N^{*(I,\underline{\a}^{\underline{t}})}_M$ of $N$ in $M$ is defined to be the submodule of $M$ consisting of all elements $z \in M$ for which there exists $\gamma \in R^{\circ, I}$ such that 
$$\gamma I^{c(q-1)}\a_1^{\lceil t_1q \rceil} \dots \a_m^{\lceil t_m q \rceil}z^q \subseteq N^{[q]}_M$$ 
for all large $q = p^e$.

\item Let $E=\oplus_{\m} E_R(R/\m)$ be the direct sum, taken over all maximal ideals $\m$ of $R$, of the injective hulls of the residue fields $R/\m$.
The \textit{generalized test ideal} $\widetilde{\tau}_I(R,\underline{\a}^{\underline{t}})$ associated to the pair $(R,\underline{\a}^{\underline{t}})$ along $I$ is 
$$\widetilde{\tau}_I(R,\underline{\a}^{\underline{t}})=\mathrm{Ann}_R(0^{*(I,\underline{\a}^{\underline{t}})}_E) \subseteq R. $$ 
We denote this ideal simply by $\widetilde{\tau}_I(R)$ when $\a_i=R$ for all  $i=1, \dots, m$. 
\end{enumerate}
\end{defn}

\begin{rem}
When $R$ is a normal domain and $I=xR$ is a principal ideal, $(I, \underline{\a}^{\underline{t}})$-tight closure coincides with divisorial $(\Div(x),\underline{\a}^{\underline{t}})$-tight closure introduced in \cite{Ta2}. 
\end{rem}

\begin{deflem}\label{test element}
Let $R$ be an excellent domain of characteristic $p>0$ and $I$ be an unmixed ideal of height $c$. 
Let $E=\bigoplus_{\m} E_R(R/\m)$ be the direct sum, taken over all maximal ideals $\m$ of $R$, of the injective hulls of the residue fields $R/\m$. Fix an element $\gamma \in R^{\circ, I}$. 
We say that $\gamma$ is an $(I,*)$-test element for $E$ if for all $\underline{\a}^{\underline{t}}=\prod_{i=1}^m \a_i^{t_i}$, where the $\a_i$ are ideals of $R$ such that $\a_i \cap R^{\circ, I} \ne \emptyset$ and the $t_i$ are positive real numbers,  one has $\gamma I^{c(q-1)}\a_1^{\lceil t_1q \rceil} \dots \a_m^{\lceil t_m q \rceil}z^q=0$ in $\F^e(E)$ for every $z \in 0^{*(I,\underline{\a}^{\underline{t}})}_E$and for every $q=p^e$. 
If $R$ is F-finite and the localized ring $R_{\gamma}$ is purely F-regular along $IR_{\gamma}$, then some power $\gamma^N$ of $\gamma$ is an $(I,*)$-test element for $E$. 
\end{deflem}
\begin{proof}
It follows from an argument similar to \cite{HH} (see also the proofs of \cite[Theorem 1.7]{HY} and \cite[Corollary 3.10 (2)]{Ta2}). 
\end{proof}

\begin{prop}\label{completion}
Let $(R,\m)$ be an F-finite local domain of characteristic $p>0$ and $I$ be an unmixed ideal of height $c$.
Let $\underline{\a}^{\underline{t}}=\prod_{i=1}^m \a_i^{t_i}$ be a formal combination, where the $\a_i$ are ideals of $R$ such that $\a_i \cap R^{\circ, I} \ne \emptyset$ and the $t_i$ are positive real numbers. 
\renewcommand{\labelenumi}{$(\arabic{enumi})$}
\begin{enumerate}
\item
Let $W$ be a multiplicatively closed subset of $R$, and $\underline{\a}_W^{\underline{t}}$ and $I_W$ be the images of $\underline{\a}^{\underline{t}}$ and $I$ in $R_W$, respectively. Then
$$\widetilde{\tau}_{I_W}(R_W,\underline{\a}_W^{\underline{t}}) =\widetilde{\tau}_I(R,\underline{\a}^{\underline{t}})R_W.$$
\item
Let $\widehat{R}$ be the $\m$-adic completion of $R$, and $\widehat{\underline{\a}}^{\underline{t}}$ and $\widehat{I}$ be the images of $\underline{\a}^{\underline{t}}$ and $I$ in $\widehat{R}$, respectively.
Then
$$\widetilde{\tau}_{\widehat{I}}(\widehat{R},\widehat{\underline{\a}}^{\underline{t}})=\widetilde{\tau}_I(R,\underline{\a}^{\underline{t}})\widehat{R}.$$
\item 
$(R,\underline{\a}^{\underline{t}})$ is purely F-regular along $I$ if and only if $\widetilde{\tau}_I(R,\underline{\a}^{\underline{t}})=R$. 
\item 
If $R$ is an F-finite regular local ring and $\gamma \in R^{\circ, I}$ is an $(I,*)$-test element for the injective hull $E_R(R/\m)$ of the residue field $R/\m$, then $\widetilde{\tau}_I(R,\underline{\a}^{\underline{t}})$ is the unique smallest ideal $J$ of $R$ with respect to inclusion, such that 
$$\gamma I^{c(q-1)}\a_1^{\lceil t_1q \rceil} \dots \a_m^{\lceil t_m q \rceil} \subseteq J^{[q]}$$
for all $($large$)$ $q=p^e$. 
\end{enumerate}
\end{prop}
\begin{proof}
(1) and (2) follow from arguments similar to the proofs of \cite[Propositions 3.1 and 3.2]{HT}, respectively. 
(3) follows from an argument similar to the proof of \cite[Proposition 2.1]{Ha1} (see also \cite[Corollary 3.5]{Ta1}) and (4) does from an argument similar to the proof of \cite[Proposition 2.22]{BMS}. 
\end{proof}

\begin{eg}
Let $(R,\m)$ be an F-finite regular local ring of characteristic $p>0$ and $I=(f_1, \dots, f_c) \subseteq R$ be an unmixed ideal generated by a regular sequence $f_1, \dots, f_c$. 
Let $\gamma \in R^{\circ, I}$ be an element such that the localized ring $R_{\gamma}/IR_{\gamma}$ is regular, and take a sufficiently large integer $N$.  
By Definition-Proposition \ref{test element} and Proposition \ref{completion} (4), $R$ is purely F-regular along $I$ if and only if there exists $q=p^e$ such that $\gamma^NI^{c(q-1)} \not\subseteq \m^{[q]}$. 
Since $(f_1 \cdots f_c)^{q-1} \in I^{c(q-1)} \subseteq (f_1 \cdots f_c)^{q-1}R+I^{[q]}=(I^{[q]}:I)$, this is equivalent to saying that there exists $q=p^e$ such that $\gamma^N(I^{[q]}:I) \not\subseteq \m^{[q]}$. 
It therefore follows from \cite[Theorem 2.1]{Gl} that $R$ is purely F-regular along $I$ if and only if $R/I$ is strongly F-regular. 
That is, the generalized test ideal $\widetilde{\tau}_I(R)$ along $I$ defines the non-strongly-F-regular locus of the ring $R/I$.  
\end{eg}

\begin{thm}\label{I < adj}
Let $(R,\m)$ be a $d$-dimensional F-finite regular local ring of characteristic $p>0$ and $I \subseteq R$ be an unmixed ideal of height $c$.
Let $\underline{\a}^{\underline{t}}=\prod_{i=1}^m \a_i^{t_i}$ be a formal combination, where the $\a_i$ are ideals of $R$ such that $\a_i \cap R^{\circ, I} \ne \emptyset$ and the $t_i$ are positive real numbers. 
Set $A=\Spec R$, $X=V(I) \subseteq A$ and $Y_i=V(\a_i) \subseteq A$. 
Let  $f: A' \to A$ be the blowing-up of $A$ along $X$ and $E_1, \dots, E_s$ be all the components of the exceptional divisor of $f$ dominating an irreducible component of $X$.  
Suppose that $\pi:\widetilde{A} \to A$ is a proper birational morphism from a normal scheme $\widetilde{A}$ such that  the scheme theoretic inverse images $\pi^{-1}(X)$ and $\pi^{-1}(Y_i)$ are Cartier divisors, and denote by $\widetilde{E}$ the strict transform of $E:=E_1+\cdots+E_S$ on $\widetilde{A}$. 
Then one has an inclusion 
$$\widetilde{\tau}_I(R,\underline{\a}^{\underline{t}}) \subseteq H^0(\widetilde{A}, \sO_{\widetilde{A}}(K_{\widetilde{A}/A}- \sum_{i=1}^m \lfloor t_i \pi^{-1}(Y_i)\rfloor -c \ \pi^{-1}(X) +\widetilde{E})).$$
\end{thm}
\begin{proof}
The proof follows from essentially the same argument as that of \cite[Proposition 3.8]{HY} (see also the proof of \cite[Proposition 3.8]{Ta1} for a different strategy). 
For simplicity, we assume that $\widetilde{A}$ is a Cohen-Macaulay scheme. 
Denote the closed fiber of $\pi$ by $Z:=\pi^{-1}(\m)$ and set $\pi^{-1}(Y):=\sum_{i=1}^m t_i \pi^{-1}(Y_i)$. 
Let 
$$\delta:H^d_{\m}(R) \to H^d_Z(\sO_{\widetilde{A}}(\lfloor \pi^{-1}(Y) \rfloor+c \ \pi^{-1}(X) -\widetilde{E}))$$ be the edge map $H^d_{\m}(R) \to H^d_{Z}(\sO_{\widetilde{A}})$ of the spectral sequence $H^i_{\m}(R^j \pi_*\sO_{\widetilde{A}}) \Rightarrow H^{i+j}_Z(\sO_{\widetilde{A}})$ followed by the natural map 
$$H^d_{Z}(\sO_{\widetilde{A}}) \to H^d_{Z}(\sO_{\widetilde{A}}(\lfloor \pi^{-1}(Y) \rfloor + c \ \pi^{-1}(X) -\widetilde{E})).$$
By the local duality theorem (see \cite[V, \S 6]{Hart}), one has
$$\Ann_R (\Ker \delta)=H^0(\widetilde{A}, \sO_{\widetilde{A}}(K_{\widetilde{A}/A}-\lfloor \pi^{-1}(Y) \rfloor - c \ \pi^{-1}(X) +\widetilde{E} )). $$
It therefore suffices to show that $\Ker \delta \subseteq 0^{*(I, \underline{\a}^{\underline{t}})}_{H^d_{\m}(R)}$. 
Take an element $\gamma \in R^{\circ, I}$ such that $R_{\gamma}/IR_{\gamma}$ is regular. 
Since $H^0(\widetilde{A}, \sO_{\widetilde{A}}(K_{\widetilde{A}/A}- c \ \pi^{-1}(X) +\widetilde{E} ))_{\gamma}=R_{\gamma}$, 
for sufficiently large integers $N \gg 0$, one has 
\begin{align*}
\gamma^N \a_1^{\lceil t_1 \rceil} \dots \a_m^{\lceil t_m \rceil} & \subseteq \a_1^{\lceil t_1 \rceil} \dots \a_m^{\lceil t_m  \rceil}H^0(\widetilde{A}, \sO_{\widetilde{A}}(K_{\widetilde{A}/A}- c \ \pi^{-1}(X) +\widetilde{E} ))\\
&\subseteq H^0(\widetilde{A}, \sO_{\widetilde{A}}(K_{\widetilde{A}/A}-\lfloor \pi^{-1}(Y) \rfloor - c \ \pi^{-1}(X) +\widetilde{E} ))\\
&=\Ann_R (\Ker \delta).
\end{align*}
By Definition-Lemma \ref{test element}, this inclusion tells us that there exists an $(I, *)$-test element $\gamma' \in \Ann_R (\Ker \delta) \cap R^{\circ, I}$ for $H^d_{\m}(R)$, because $\a_i \cap R^{\circ, I} \ne \emptyset$ for all $i=1, \dots, m$. 
For every $q=p^e$ and for every 
$$\alpha \in I^{c(q-1)}\a_1^{\lceil t_1q \rceil} \dots \a_m^{\lceil t_m q \rceil}  \subseteq H^0(\widetilde{A}, \sO_{\widetilde{A}}(-q\lfloor \pi^{-1}(Y) \rfloor - c(q-1)\pi^{-1}(X))),$$ 
we have the following commutative diagram with exact rows:
\[
\xymatrix{
 0 \ar[r] & \Ker \delta \ar[d] \ar[r] & H^d_{\m}(R) \ar[d]^{\alpha F^e} \ar[r]^{\hspace*{-7.5em}\delta} &H^d_{Z}(\sO_{\widetilde{A}}(\lfloor \pi^{-1}(Y) \rfloor + c \ \pi^{-1}(X) -\widetilde{E}))) \ar[d]^{\alpha F^e} \ar[r] & 0\\
0 \ar[r] & \Ker \delta \ar[r] & H^d_{\m}(R) \ar[r]^{\hspace*{-7.5em}\delta} & H^d_{Z}(\sO_{\widetilde{A}}(\lfloor \pi^{-1}(Y) \rfloor + c \ \pi^{-1}(X) -\widetilde{E})))  \ar[r] & 0 \\
}
\]
Then $\alpha F^e(\Ker \delta) \subseteq \Ker \delta$. 
By the choice of the element $\gamma'$, we can conclude that $\gamma' I^{c(q-1)}\a_1^{\lceil t_1q \rceil} \dots \a_m^{\lceil t_m q \rceil} F^e(\Ker \delta)=0$ for all $q=p^e$, that is, $\Ker \delta \subseteq 0^{*(I, \underline{\a}^{\underline{t}})}_{H^d_{\m}(R)}$. 
\end{proof}

We conjecture that the generalized test ideal $\widetilde{\tau}_I(R,\underline{\a}^{\underline{t}})$ along $I$ corresponds to the adjoint ideal sheaf $\adj_X(A,Y)$. 
\begin{conj}\label{correspond conj}
Let $(R,\m)$ be a regular local ring essentially of finite type over a perfect field of prime characteristic $p$, and let $I \subseteq R$ be a nonzero unmixed ideal. 
Let $\underline{\a}^{\underline{t}}=\prod_{i=1}^m \a_i^{t_i}$ be a formal combination, where the $\a_i$ are ideals of $R$ such that $\a_i \cap R^{\circ, I} \ne \emptyset$ and the $t_i$ are positive real numbers.
Set $A:=\Spec R$, $X:=V(I)$ and $Y:=\sum_{i=1}^m t_i V(\a_i)$.  
Assume in addition that $(R,I, \underline{\a})$ is reduced from characteristic zero to characteristic $p \gg 0$, together with a log resolution $\pi:\widetilde{A} \to A$ of $(A,X+Y)$ used to define the adjoint ideal sheaf $\adj_X(A,Y)$ as in Definition \ref{adjoint def}. Then
$$\adj_X(A,Y)=\widetilde{\tau}_{I}(R,\underline{\a}^{\underline{t}}).$$
\end{conj}

Conjecture \ref{correspond conj} is true if $X$ is a divisor on $A$. 
\begin{thm}[\textup{\cite[Theorem 5.3]{Ta2}}]\label{correspond}
Let $(R,\m)$ be a $\Q$-Gorenstein normal local ring essentially of finite type over a perfect field of prime characteristic $p$, and let $f$ be a nonzero element of $R$. 
Let $\underline{\a}^{\underline{t}}=\prod_{i=1}^m \a_i^{t_i}$ be a formal combination, where the $\a_i$ are ideals of $R$ such that $\a_i \cap R^{\circ, fR} \ne \emptyset$ and the $t_i$ are positive real numbers.
Set $X=\Spec R$, $D:=\Div_X(f)$ and $Y:=\sum_{i=1}^m t_i V(\a_i)$.  
Assume in addition that $(R,f, \underline{\a})$ is reduced from characteristic zero to characteristic $p \gg 0$, together with a log resolution $\pi:\widetilde{X} \to X$ of $(X,D+Y)$ used to define the adjoint ideal sheaf $\adj_D(X,Y)$ as in Definition \ref{codimension one}. Then
$$\adj_D(X,Y)=\widetilde{\tau}_{fR}(R,\underline{\a}^{\underline{t}}).$$
\end{thm}

\section{Restriction formula of adjoint ideals}
In this section, we formulate a restriction property of the adjoint ideal sheaf $\adj_X(A, Y)$ involving the l.c.i.~defect ideal sheaf $\mathcal{D}_X$ of $X$. 

Let $A$ be a nonsingular variety over an algebraically closed field $k$ of characteristic zero and $X$ be a normal Gorenstein closed subvariety of codimension $c$ of $A$. 
Kawakita \cite{Ka} then defined the \textit{l.c.i.~defect ideal sheaf} $\mathcal{D}_X$ of $X$ as follows. 
Since the construction is local, we may consider the germ at a closed point $x \in X$. 
We take generically a closed subscheme $Z$ of $A$ which contains $X$ and is locally a complete intersection of codimension $c$. 
By Bertini's theorem, $Z$ is the scheme-theoretic union of $X$ and another variety $C^Z$ of codimension $c$. 
Since $X$ is Gorenstein, the closed subscheme $D^Z:=C^Z|_X$ of $X$ is a Cartier divisor (see \cite[Lemma 1]{Vr}).  
Then the l.c.i.~defect ideal sheaf $\mathcal{D}_X$ of $X$ is defined by
$$\mathcal{D}_X:=\sum_{Z \subset A}\sO_X(-D^Z),$$
where $Z$ runs through all the general locally complete intersection closed subschemes of codimension $c$ which contain $X$. 
Note that the support of $\mathcal{D}_X$ coincides with the non-locally complete intersection locus of $X$. 
The reader is referred to \cite[Section 2]{Ka} and \cite[Section 9.2]{EM} for further properties of l.c.i.~defect ideal sheaves. 
 
\begin{thm}\label{restriction}
Let $A$ be a nonsingular variety over an algebraically closed field $k$ of characteristic zero and $Y=\sum_{i=1}^m t_i Y_i$ be a formal combination, where the $t_i$ are positive real numbers and the $Y_i$ are proper closed subschemes of $A$. If $X$ is a normal Gorenstein closed subvariety of codimension $c$ of $A$ which is not contained in the support of any $Y_i$, then 
$$\J(X,V(\mathcal{D}_X)+Y|_X)=\adj_X(A,Y) \sO_X,$$
where $\mathcal{D}_X$ is the l.c.i.~defect ideal sheaf of $X$.  
\end{thm}

\begin{proof}
Since the question is local, we consider the germ at a closed point $x \in X \cap \bigcap_{i=1}^m Y_i \subset A$. Let $\I_X \subseteq \sO_A$ be the defining ideal sheaf of $X$ in $A$. 

Fix a regular function $\varphi \in \adj_X(A,Y) \setminus \I_X$. 
It then follows from Lemma \ref{adjoint basic} (2) that $\mathrm{mld}(X_{\rm sing} \cup \bigcup_{i=1}^m Y_i; A,cX+Y-\Div_A(\varphi))>0$. 
Applying \cite[Remark 8.5]{EM} (see also the proof of \cite[Theorem 1.1]{Ka}), we have 
$$\mathrm{mld}(X_{\rm sing} \cup \bigcup_{i=1}^m (X \cap Y_i); X,V(\mathcal{D}_X)+Y|_X-\Div_X(\overline{\varphi}))>0,$$
where $\overline{\varphi}$ is the image of $\varphi$ in $\sO_X$. 
This means that $(X,V(\mathcal{D}_X)+Y|_X-\Div_X(\overline{\varphi}))$ is klt, which is equivalent to saying that $\overline{\varphi}$ is in $\J(X,V(\mathcal{D}_X)+Y|_X)$. 
Thus, we conclude that 
$$\adj_X(A,Y) \sO_X \subseteq \J(X,V(\mathcal{D}_X)+Y|_X).$$

Next we will prove the converse inclusion. 
Take generically a closed subscheme $Z$ of $A$ which contains $X$ and is locally a complete intersection
of codimension $c$, so $Z$ is the scheme-theoretic union of $X$ and another variety $C^{Z}$, and $D^{Z}:=C^{Z}|_X$ is a Cartier divisor on $X$. 
\begin{cln}
By a general choice of $Z$,  one has 
$$\J(X,V(\mathcal{D}_X)+Y|_X)=\adj_{D^{Z}}(X,Y|_X).$$
\end{cln}
\begin{proof}[Proof of Claim 1 $($Kawakita$)$]
Since $\J(X,V(\mathcal{D}_X)+Y|_X) \supseteq \adj_{D^{Z}}(X,Y|_X)$ is clear from the definition of the ideal sheaf $\mathcal{D}_X$, we will prove the converse inclusion. 

Fix a regular function $\psi \in \J(X,V(\mathcal{D}_X)+Y|_X)$. 
By the definition of the ideal sheaf $\mathcal{D}_X$, there exist closed subschemes $W_1, \dots, W_n$ of $A$ which  contain $X$ and are locally complete intersections of codimension $c$ such that $\mathcal{D}_X=\sum_{j=1}^n \sO_X(-D^{W_j})$.  
Take a log resolution $\mu:\widetilde{X} \to X$ of $(X, D^{W_1}+\dots+D^{W_n}+Y|_X)$, and let $\{E_i\}_{i \in I}$ be a collection of all divisors on $\widetilde{X}$ which are supported on $\mathrm{Exc}(\mu) \cup \bigcup_{i=1}^m \mathrm {Supp}\;\mu^{-1}(Y_i|_X) \cup \bigcup_{j=1}^n \mu^{-1}_*D^{W_j}$. 
Since $\psi$ is in $\J(X,V(\mathcal{D}_X)+Y|_X)$, we have 
$$\mathrm{ord}_{E_i}(\psi)+\max_{1 \le j \le n} a(E_i; X, D^{W_j}+Y|_X) > 0$$ 
for all $i \in I$. 
Let $I_{W_j}=(f_1^{(j)}, \dots, f_c^{(j)}) \subseteq \sO_A$ be the defining ideal sheaf of $W_j$ in $A$, and set $I_{\mathcal{W}}:=(t_1 f^{(1)}_1+\cdots+t_n f^{(n)}_1, \dots, t_1 f^{(1)}_c+\cdots+t_n f^{(n)}_c) \subseteq \sO_A[t_1, \dots, t_n]$. 
Let $\mathcal{W} \subseteq A \times T$ be the corresponding closed subscheme, where $T:=\Spec k[t_1, \dots, t_n]$. 
Then $\mathcal{W}$ is the scheme-theoretic union of $X \times T$ and another variety $\mathcal{C}$. 
Note that $\mathcal{C}|_{X \times T}$ is an irreducible Cartier divisor over a generic point of $T$. 
One can choose a generator $h \in \sO_X \otimes k(t_1, \dots, t_n)$ of the principal ideal sheaf $\sO_{X \times T}(-\mathcal{C}|_{X \times T})$ over a generic point of $T$ so that the restriction of $h$ to the fiber over $(0,\ldots,0,\stackrel{j}{\check{1}},0,\ldots,0) \in T$ is a generator of $\sO_X(-D^{W_j})$.
Thinking  of the resolution $\mu \times {\rm id}_T: \widetilde{X} \times T \to X \times T$ induced by  $\mu$, we have 
$$a(\mathcal{E}_i; X \times T, \Div_{}(h)+Y|_X \times T) \ge a(E_i; X, D^{W_j}+Y|_X)$$
for all $i \in I$ and all $j=1, \dots, n$, where $\mathcal{E}_i:=E_i \times T \subset \widetilde{X} \times T$. 

On the other hand, the restriction of $h$ to the fiber over a general point $(t_1, \dots, t_n) \in T$ is a generator of $\sO_X(-D^{W_{t_1, \dots, t_n}})$ where $W_{t_1, \dots, t_n}:=\mathcal{W}|_{X \times (t_1, \dots, t_n)}$, and $\mu$ is a log resolution of $(X, W_{t_1, \dots, t_n}+Y|_X)$. Thus, for all $i \in I$, 
\begin{align*}
\mathrm{ord}_{E_i}(\psi)+a(E_i; X, D^{W_{t_1, \dots, t_n}}+Y|_X)&=\mathrm{ord}_{E_i}(\psi)+a(\mathcal{E}_i; X \times T, \Div_{}(h)+Y|_X \times T)\\
& \ge \mathrm{ord}_{E_i}(\psi)+\max_{1 \le j \le n} a(E_i; X, D^{W_j}+Y|_X)\\
& > 0.
\end{align*}
This implies that $\psi$ is in $\adj_{D^{W_{t_1, \dots, t_n}}}(X,Y|_X)$. 
\end{proof}
From now on, we may assume that $A=\Spec S$ and $X=\Spec R$, where $(S,\n)$ is a regular local ring essentially of finite type over a field of characteristic zero and and $R=S/I$ is a Gorenstein normal quotient of $S$. Let $\a_i$ be the ideal of $S$ corresponding to $Y_i$ for every $i=1, \dots, n$ and denote $\underline{\a}^{\underline{t}}=\prod \a_i^{t_i}$.  
Let $f_1, \dots, f_c$ be the regular sequence in $S$ corresponding to $Z$ and $f \in S$ be an element whose image $\overline{f}$ is a generator of the principal ideal $(((f_1, \dots, f_c):I)+I)/I$ of $R$.  
Thanks to Claim $1$, it is enough to prove that 
\begin{equation}
\adj_{\Div(\overline{f})}(X,Y|_X) \subseteq \adj_X(A,Y) \sO_X.
\end{equation}

Now we reduce the entire setup as above to characteristic $p \gg 0$ and switch the notation to denote things after reduction modulo $p$. 
In order to prove (1), by virtue of Theorems \ref{I < adj} and \ref{correspond}, it suffices to show that 
\begin{equation}
\widetilde{\tau}_{fR}(R, {\underline{\a}R}^{\underline{t}}) \subseteq \widetilde{\tau}_I(S,\underline{\a}^{\underline{t}})R.
\end{equation}
Since $S$ is F-finite, by Lemma \ref{completion}, forming generalized test ideals commutes with completion. 
Hence, we may assume that $S$ is complete. 
Let $E_S=E_S(S/\n)$ and $E_R=E_R(R/\n R)$ be the injective hulls of the residue fields of $S$ and $R$, respectively. 
We can view $E_R$ as a submodule of  $E_S$ via the isomorphism $E_R \cong (0:I)_{E_S} \subset E_S$. 
Then by Matlis duality,  (2) is equivalent to saying that
\begin{equation}
0_{E_R}^{*(fR, {\underline{\a}R}^{\underline{t}})} \supseteq 0_{E_S}^{*(I, {\underline{\a}}^{\underline{t}})} \cap E_R.
\end{equation}

Let $z \in  0_{E_S}^{*(I, {\underline{\a}}^{\underline{t}})} \cap E_R$.
Let $F^e_S:E_S \to \F^e_S(E_S) \cong E_S$ and $F^e_R:E_R \to \F^e_R(E_R) \cong E_R$ be the $e$-times iterated Frobenius maps induced on $E_S$ and $E_R$, respectively. 
Since $R=S/I$ is normal, one can choose an element $\gamma \in S^{\circ, I}$ such that the image $\overline{\gamma}$ of $\gamma$ is not contained in any minimal prime of $fR$ and the localized ring $R_{\overline{\gamma}}$ is regular. 
By Definition-Lemma \ref{test element}, some power $\gamma^N$ of $\gamma$ is an $(I,*)$-test element for $E_S$, and then $\gamma^N I^{c(q-1)}\a_1^{\lceil t_1q \rceil} \dots \a_m^{\lceil t_m q \rceil}F^e_S(z)=0$ for all $q=p^e$. 
On the other hand, $I^{[q]}F^e_S(z)=0$ for all $q=p^e$, because $z \in E_R \cong (0:I)_{E_S}$. 

\begin{cln} 
For all $q=p^e$, one has 
$$f^{q-1}(I^{[q]}:I) \subseteq  I^{c(q-1)}+I^{[q]}.$$
\end{cln}

\begin{proof}[Proof of Claim 2]
This claim is an easy consequence of the linkage theory of Peskine and Szpiro \cite{PS}. 
We consider a simultaneous minimal free resolution of a natural diagram between $S/(f_1, \dots, f_c), S/I, S(f_1^q, \dots, f_c^q)$ and  $S/I^{[q]}$. 
\[
\xymatrix{
& 0 \ar[rr] & & \save[]+<0.03cm,-0.3cm> S \ar@<-0.5ex>[ddl]|(.45)\hole \restore \ar[r] \ar@<0.3ex>[d]^{\times (f_1\cdots f_c)^{q-1}} & \dots \ar[r] & \save[]+<-1.3cm,-0.25cm> S/(f_1^q,\dots,f_c^q) \ar[ddl]|(.45)\hole \restore \ar[r] \ar@<-6.5ex>[d] & 0 \\
& 0 \ar[rr] & & \save[]+<0.03cm,-0.3cm> S \ar@<-0.5ex>[ddl]^(.2){\times f} \restore \ar[r] & \dots \ar[r] & \save[]+<-1.3cm,-0.25cm> S/(f_1,\dots,f_c) \ar[ddl] \restore \ar[r] & 0 \\
0 \ar[rr] & & S \ar[r]|\hole \ar@<-0.3ex>[d] & \dots \ar[r] & S/I^{[q]} \ar[r]|(.28)\hole \ar[d] & 0 & \\
0 \ar[rr] & & S \ar[r] & \dots \ar[r] & S/I \ar[r] & 0 & 
}
\]
Here note that the ideals $I, I^{[q]}, (f_1, \dots, f_c)$ and $(f_1^q, \dots, f_c^q)$ are all Gorenstein of height $c$.  
Looking at the last step of the above diagram, by \cite[Lemma 1]{Vr}, we obtain the equality $f^q(I^{[q]}:I)=f (f_1 \cdots f_c)^{q-1}$ in  $S/I^{[q]}$. Since $f$ is a regular element of $S/I$, by the flatness of the Frobenius map, $f$ is also a regular element of $S/I^{[q]}$. 
Therefore, $f^{q-1} (I^{[q]}:I)=(f_1 \cdots f_c)^{q-1}$ in $S/I^{[q]}$, which gives the assertion.  

\end{proof}

Thanks to Claim 2, we have $\gamma^N f^{q-1} (I^{[q]}:I) \a_1^{\lceil t_1q \rceil} \dots \a_m^{\lceil t_m q \rceil}F^e_S(z)=0$ for all $q=p^e$. 
This is equivalent to saying that $\overline{\gamma}^N \overline{f}^{q-1}(\a_1R)^{\lceil t_1q \rceil} \dots (\a_mR)^{\lceil t_m q \rceil}F^e_R(z)=0$ for all $q=p^e$, because $\F_R^e(E_R) \cong (0:I^{[q]})_{E_S}/(0:(I^{[q]}:I))_{E_S}$ (see \cite{Fe}, \cite{Gl} and the proof of \cite[Lemma 3.9]{Ta1}). 
Since $\overline{\gamma}$ is not in any minimal prime of $fR$, we conclude that $z \in 0_{E_R}^{*(fR, {\underline{\a}R}^{\underline{t}})}$. 
\end{proof}

\begin{rem}
(1)
In the proof of Theorem \ref{restriction}, we have used \cite[Remark 8.5]{EM} which was originally proved by using the theory of jet schemes. It, however, can be proved without using the theory of jet schemes (see the proof of \cite[Theorem 1.1]{Ka}). So, our proof does not rely on the theory of jet schemes. 

(2)
In the statement of \cite[Theorem 1.1]{EM}, the coefficients of $Y$ have to be nonnegative. Therefore, Theorem \ref{restriction} does not  follow from their result. 

(3)
The l.c.i.~defect ideal sheaf can be defined even when $X$ is only $\Q$-Gorenstein. 
Even in this case, Kawakita \cite{Ka} and Ein-Musta\c t\v a \cite{EM} formulated a comparison of minimal log discrepancies of $X$ and $A$. We, therefore, expect a generalization of Theorem \ref{restriction} to the $\Q$-Gorenstein case, but it is open.  
\end{rem}

\begin{eg}
(1) Let $X=\frac{1}{3}(1,1,1)$ be the quotient of $\C^3=\Spec \C[x_1, x_2, x_3]$ by the action of $\Z/ 3\Z$ given by $x_i \mapsto \xi x_i$, where $\xi$ is a primitive cubic root of unity. 
Denote by $\m_{X,0} \subset \sO_X$ (resp. $\m_{\C^3,0} \subset \sO_{\C^3}$) the maximal ideal sheaf of the origin.
By \cite[Example 2.3]{Ka}, the integral closure of the l.c.i.~defect ideal sheaf $\mathcal{D}_X$ of $X$ is $\m_{X,0}^5$. 
Therefore,
\begin{align*}
\J(X,V(\mathcal{D}_X))=\J(X, \m_{X,0}^5)&=\J(\C^3, \m_{\C^3,0}^{15}) \cap \sO_X\\
&=\m_{\C^3,0}^{13} \cap \sO_X\\
&=\m_{X,0}^5.
\end{align*} 
On the other hand, $X$ can be embedded into $A:=\C^{10}$. Then we have already seen in Example \ref{adjoint example} (2) that $\adj_X(A)=\m_{A,0}^5$, where $\m_{A,0} \subset \sO_A$ is the maximal ideal sheaf of the origin. Thus, $\J(X,V(\mathcal{D}_X))=\adj_X(A) \sO_X=\m_{X,0}^5$. 

(2) Let $X$ be the quotient of $(x^2+y^3+z^6=0) \subset \C^3$ by the action of $\Z/ 5\Z$ given by $x \mapsto \xi^3  x$, $y \mapsto \xi^2  y$ and $z \mapsto \xi  z$, where $\xi$ is a primitive quintic root of unity. Then $X$ can be embedded into $A:=\C^5$. 
Since $X$ is a Gorenstein closed subvariety of codimension three of $A$,  
by \cite[Example 2.4]{Ka} (which is an application of the structure theorem for Gorenstein ideals of codimension three in \cite{BE}), the l.c.i.~defect ideal sheaf $\mathcal{D}_X$ of $X$ is the maximal ideal sheaf $\m_{X,0}$ of the origin. 
Let $\pi:\widetilde{X} \to X$ be the blowing-up of $X$ at the origin with exceptional divisor $E$. Then $\pi$ is a log resolution of $X$ and $K_{\widetilde{X}/X}=-E$. 
Therefore, 
$$\J(X, V(\mathcal{D}_X))=\pi_*\sO_{\widetilde{X}}(K_{\widetilde{X}/X}-E)=\pi_*\sO_{\widetilde{X}}(-2E)=\m_{X,0}^2.$$
On the other hand, we have already seen in Example \ref{adjoint example} (3) that $\adj_X(A)=\m_{A,0}^2$, where $\m_{A,0} \subset \sO_A$ is the maximal ideal sheaf of the origin. 
Thus, $\J(X,V(\mathcal{D}_X))=\adj_X(A) \sO_X=\m_{X,0}^2$. 
\end{eg}

We conclude this section by stating a corollary of Theorem \ref{restriction}. 

Let $S$ be an algebra essentially of finite type over a field $k$ of characteristic zero, and let $I \subseteq S$ be an unmixed ideal of height $c$.  
Let $\underline{\a}^{\underline{t}}=\prod_{i=1}^m \a_i^{t_i}$ be a formal combination, where the $\a_i$ are ideals of $S$ such that $\a_i \cap S^{\circ, I} \ne \emptyset$ and the $t_i$ are positive real numbers.
We say that $(S, \underline{\a}^{\underline{t}})$ is \textit{of purely F-regular type} along $I$ if there exist a finitely generated $\Z$-subalgebra $A$ of $k$ and an algebra $S_A$ essentially of finite type over $A$ satisfying the following conditions:
\renewcommand{\labelenumi}{(\roman{enumi})}
\begin{enumerate}
\item $S_A$ is  flat over $A$. In addition,  $S_A \otimes_A k \cong S$, $I_A S=I$ and $\underline{\a}_A S=\underline{\a}$, where $I_A:=I \cap S_A \subseteq S_A$ and $\underline{\a}_A$ is the restriction of $\underline{\a}$ to $S_A$. 
\item $(S_{\kappa}, \underline{\a}_{\kappa}^{\underline{t}})$ is purely F-regular along $I_{\kappa}$ for every closed point $s$ in a dense open subset of $\Spec A$, where $\kappa=\kappa(s)$ denotes the residue field of $s \in \Spec A$, $S_{\kappa}=S_A \otimes_A \kappa(s)$, $I_{\kappa}:=I_A S_{\kappa}$ and $\underline{\a}_{\kappa}$ is the image of $\underline{\a}_A$ in $S_{\kappa}$. 
\end{enumerate}

\begin{cor}\label{purely F-regular vs plt}
In the above situation, suppose in addition that $k$ is an algebraically closed field, $S$ is a regular domain and $S/I$ is a Gorenstein quotient of $S$. 
Then $(\Spec S,\sum_{i=1}^m t_i V(\a_i))$ is plt along $V(I)$ if and only if $(S, \underline{\a}^{\underline{t}})$ is of purely F-regular type along $I$. 
\end{cor}

\begin{proof}
Since the statement is local, we may assume that $S$ is a regular local ring. The ``if" part immediately follows from Proposition \ref{completion} (3) and Theorem \ref{I < adj}. We will prove the ``only if" part. 

Suppose that $(\Spec S,\sum_{i=1}^m t_i V(\a_i))$ is plt along $V(I)$. 
Then $R:=S/I$ is a normal local ring. 
Let $f_ 1, \dots, f_c$ be a general regular sequence in $S$ and $f \in S$ be an element whose image 
$\overline{f}$ is a generator of the principal ideal $(((f_1, \dots, f_c):I)+I)/I$ of $R$. 
Then, by Theorem \ref{restriction} and Claim $1$ in the proof of Theorem \ref{restriction}, the pair $(\Spec R, \sum_{i=1}^m t_i V(\a_i R))$ is plt along $\Div(\overline{f})$. 
Thanks to Theorem \ref{correspond}, this implies that $(R, {\underline{\a}R}^{\underline{t}})$ is of purely F-regular type along $fR$. 
Applying the argument used to prove the inclusion (2) in the proof of Theorem \ref{restriction}, we know that $(S, \underline{\a}^{\underline{t}})$ is of purely F-regular type along $I$.  
\end{proof}

\begin{small}
\begin{acknowledgement}
The author is indebted to Masayuki Kawakita for providing the proof of Claim 1 in Theorem \ref{restriction} and for answering several questions.
He is also grateful to Mircea Musta\c t\v a for sending him his preprint prior to publication and for helpful advice. 
He would like to thank Craig Huneke, Yasunari Nagai, Karl Schwede, Takafumi Shibuta and Kei-ichi Watanabe for valuable discussions and Natsuo Saito for drawing the diagram in the proof of Theorem \ref{restriction}. 
The author was partially supported by Grant-in-Aid for Young Scientists (B) 17740021 from JSPS and by Program for Improvement of Research Environment for Young Researchers from SCF commissioned  by MEXT of Japan. 
\end{acknowledgement}
\end{small}

\end{document}